\documentclass[12pt,oneside,english]{amsart}
\usepackage[LGR,T1]{fontenc}
\usepackage[latin9]{inputenc}
\usepackage{geometry}
\usepackage{color}
\usepackage{amsthm}
\usepackage{amstext}
\usepackage{amssymb}
\usepackage{esint}

\makeatletter

\DeclareFontEncoding{LGR}{}{}
\DeclareTextSymbol{\~}{LGR}{126}

\numberwithin{equation}{section}
\numberwithin{figure}{section}
  \theoremstyle{plain}
  \newtheorem*{thm*}{\protect\theoremname}
\theoremstyle{plain}
\newtheorem{thm}{Theorem}[section]
\newtheorem{lem}[thm]{Lemma}
  \theoremstyle{remark}
  \newtheorem*{rem*}{\protect\remarkname}
  \theoremstyle{plain}
  \theoremstyle{remark}

\makeatother

\usepackage{babel}
  
  \providecommand{\remarkname}{Remark}
  \providecommand{\theoremname}{Theorem}
\providecommand{\theoremname}{Theorem}

\begin{document}
\sloppy

\title{A model space approach to some classical inequalities for rational
functions}

\author{Anton Baranov}

\address{Department of Mathematics and Mechanics, Saint Petersburg State University,
28, Universitetski pr., St. Petersburg, 198504, Russia}

\email{anton.d.baranov@gmail.com}

\author{Rachid Zarouf}

\address{Aix-Marseille Universit\'e, CNRS, Centrale Marseille, LATP, UMR 7353,
13453 Marseille, France.}

\email{rachid.zarouf@univ-amu.fr}

\thanks{The first author is supported by the Chebyshev Laboratory (St. Petersburg
State University) 
under RF Government grant 11.G34.31.0026, 
by JSC "Gazprom Neft" and by RFBR grant 12-01-00434.}
\begin{abstract}
We consider the set $\mathcal{R}_{n}$ of rational functions of degree
at most $n\geq1$ with no poles on the unit circle $\mathbb{T}$
and its subclass $\mathcal{R}_{n,\, r}$
consisting of rational functions without poles in the annulus 
$\left\{ \xi:\; r\leq|\xi|\leq\frac{1}{r}\right\}$. 
We discuss an approach based on the model space theory which brings some
integral representations for functions in $\mathcal{R}_{n}$ and their
derivatives. Using this approach we obtain $L^p$-analogs
of several classical inequalities for rational functions
including the inequalities by P. Borwein and T. Erd\'elyi, the Spijker Lemma
and S.M. Nikolskii's inequalities. These inequalities are shown
to be asymptotically sharp as $n$ tends to infinity and the poles
of the rational functions approach the unit circle $\mathbb{T}.$ 
\end{abstract}
\maketitle

\section{\label{Sec:Introduction}Introduction}

The goal of this paper is to give a unified approach
to several classical inequalities for rational functions. 
This approach is based on integral representations for rational 
functions and their derivatives. It makes possible to recover
several known results and obtain their $L^p$ analogs where
the estimate is given not only in terms of the degree 
of a rational function but also in terms of the distance from 
the poles to the boundary.

\subsection{Notations}

Let $\mathcal{P}_{n}$ be the space of complex analytic polynomials
of degree at most $n\geq1$ and let 
\[
\mathcal{R}_{n}=\left\{ \frac{P}{Q}\,:\; P,\, Q\in\mathcal{P}_{n},\; Q(\xi)\ne0\;\;
\text{for}\;\; |\xi|=1\right\} ,
\]
be the set of rational functions of degree at most $n$ (where $\deg\,\frac{P}{Q}=\max\left(\deg\, P,\,\deg\, Q\right)$)
without poles on the unit circle 
$\mathbb{T}=\left\{ \xi\in\mathbb{C}:\;|\xi|=1\right\}$.
We denote by $\|f\|_{L^p}$, $1\le p\le \infty$, the standard norms of the spaces
$L^p(\mathbb{T}, {m})$, where ${m}$
stands for the normalized Lebesgue measure on $\mathbb{T}$.
Denote by $\mathbb{D}=\left\{ \xi\in\mathbb{C}:\;|\xi|<1\right\} $
the unit disc of the complex plane and by $\overline{\mathbb{D}}=
\left\{ \xi\in\mathbb{C}:\;|\xi|\leq1\right\} $
its closure. For a given $r\in(0,\,1),$ we finally introduce the
subset 
\[
\mathcal{R}_{n,\, r}=\left\{ \frac{P}{Q}\,:\; P,\, Q\in\mathcal{P}_{n},\; 
Q(\xi)\ne0\;\; \text{for}\;\; r\leq|\xi|\leq\frac{1}{r}\right\}
\]
of $\mathcal{R}_{n},$ consisting of rational functions of degree
at most $n$ without poles in the annulus 
$\left\{ \xi:\; r\leq|\xi|\leq\frac{1}{r}\right\}$. 

We also introduce some notations specific to the theory
of model subspaces of the Hardy space $H^{p}$, $1\leq p\leq\infty$.
Denote by ${\rm Hol}\left(\mathbb{D}\right)$ the space of all
holomorphic functions on $\mathbb{D}$. The Hardy space $H^{p}=H^{p}(\mathbb{D}),$
$1\leq p<\infty$, is defined as follows: 
\[
H^{p}=\left\{ f\in{\rm Hol}\left(\mathbb{D}\right):\:\left\Vert f\right\Vert _{H^{p}}^{p}=\sup_{0\leq\rho<1}\int_{\mathbb{T}}\left|f(\rho\xi)\right|^{p}{\rm d}{m}(\xi)<\infty\right\} .
\]
As usual, we denote by $H^{\infty}$ the space of all bounded analytic
functions in $\mathbb{D}$. For any $\sigma=(\lambda_{1},\dots,\lambda_{n})\in\mathbb{D}^{n}$,
we consider the finite Blaschke product 
\[
B_{\sigma}=\prod_{k=1}^{n}b_{\lambda_{k}},
\qquad b_{\lambda}(z)=\frac{\lambda-z}{1-\overline{\lambda}z},
\]
$b_{\lambda}$ being the
elementary Blaschke factor associated to $\lambda\in\mathbb{D}$.
Define the model subspace $K_{B_{\sigma}}$ of the Hardy space $H^{2}$
by 
\[
K_{B_{\sigma}}=\left(B_{\sigma}H^{2}\right)^{\perp}=H^{2}\ominus B_{\sigma}H^{2}.
\]
The subspace $K_{B_{\sigma}}$ consists of rational functions of the
form $P/Q$, where $P\in\mathcal{P}_{n-1}$ and $Q$ is a polynomial of degree $n$
with the zeros $1/\overline{\lambda}_{1},\dots,1/\overline{\lambda}_{n}$
of corresponding multiplicities. 

For any Blaschke product $B$, the reproducing kernel of the model
space $K_{B}$ corresponding to a point $\xi\in\mathbb{D}$ is of
the form 
\[
k_{\xi}^{B}(z)=\frac{1-\overline{B(\xi)}B(z)}{1-\overline{\xi}z}.
\]


\subsection{Some classical inequalities for rational functions}

In this subsection we give a brief review of several well-known 
inequalities for polynomials and rational functions.

\subsubsection{Pointwise estimates for the derivatives of the functions 
in $\mathcal{R}_{n}$}

Let us start with the following theorem.
\begin{thm*}
\label{ThLR} For any function $f\in\mathcal{R}_{n}$ 
with the poles $\{a_{k}\}$ \textup(counting multiplicities\textup) we have 
\begin{equation}
|f'(\xi)|\le\big\| f\big\|_{L^{\infty}}
\bigg(\sum_{\left|a_{k}\right|<1}\frac{1-\left|a_{k}\right|^{2}}{\left|a_{k}-\xi\right|^{2}}
+ \sum_{\left|a_{k}\right|>1}\frac{\left|a_{k}\right|^{2}-1}{\left|a_{k}-\xi\right|^{2}}
\bigg),
\qquad|\xi|=1\label{LR}.
\end{equation}
 
\end{thm*}
Inequality (\ref{LR}) has a long history. It was for the first time
explicitly stated and proved (by two different methods) for the case
when all poles are outside $\overline{\mathbb{D}}$ in a monograph
by V.N. Rusak \cite[Chapter III, Section 1]{rus} in 1979. 
Also, as Rusak mentions, this inequality is contained (only with
a hint of a proof) in the book of V.I. Smirnov and N.A. Lebedev
\cite[Chapter V, Section 3, Corollary 3]{sl}. 

At the same time this inequality (for poles both inside $\mathbb{D}$
and outside $\overline{\mathbb{D}}$) is a very special case of
results of M.B. Levin \cite{lev,lev1} which were obtained already
in 1974--1975, but remained unnoticed; these results apply to arbitrary
functions admitting pseudocontinuation.

Further extensions of Levin--Rusak inequality were obtained in the
1990s independently by two groups of specialists in polynomial inequalities
\cite{erd1,lmr}. In particular, in \cite{be} (see also \cite[Theorem 7.1.7]{erd1})
P. Borwein and T. Erd\'elyi obtained the following interesting improvement
which shows that the sum in \textup{(\ref{LR})}
may be replaced by the maximum. 

\begin{thm*}
\label{ThBE} 
For any function $f\in\mathcal{R}_{n}$ 
with the poles $\{a_{k}\}$ we have
\begin{equation}
\label{BE}
|f'(\xi)|\le\big\| f\big\|_{L^{\infty}}\max\bigg(\sum_{\left|a_{k}\right|<1}
\frac{1-\left|a_{k}\right|^{2}}{\left|a_{k}-\xi\right|^{2}},\,\,
\sum_{\left|a_{k}\right|>1}
\frac{\left|a_{k}\right|^{2}-1}{\left|a_{k}-\xi\right|^{2}} \bigg),
\qquad|\xi|=1.
\end{equation}
\end{thm*}
\medskip


\subsubsection{Spijker's Lemma}

A well-known result by M.N. Spijker \cite{Sp} 
(known as Spijker's Lemma)  asserts
that the image of $\mathbb{T}$ under a complex rational map $f\in\mathcal{R}_{n}$
has length at most $2n\pi$, or, in other words,
\begin{equation}
\label{dolz}
\|f'\|_{L^1} \le n \|f\|_\infty, \qquad f\in \mathcal{R}_{n}.
\end{equation}
The inequality is apparently sharp (take $f(z) = z^n$). This 
result published by Spijker in 1991 ended a long search for the best 
bound in this inequality 
(e.g., in 1984 R.J. Leveque and L.N. Trefethen \cite{LeTr} 
proved the above inequality with $2n$ in place of $n$).
The importance of the sharp constants in this inequality is related to its role
in the \textit{Kreiss Matrix Theorem} \cite{Kr}.

However, it was recently noticed (see \cite{NN}) that
inequality \eqref{dolz} was discovered already in 1978 by
E.P. Dolzhenko  \cite{Dol} as a special case 
of more general results. Unfortunately, this paper (which appeared only 
in Russian) remained unknown to the specialists. Let us cite the following
beautiful theorem from \cite{Dol} where majorization on the whole circle 
is replaced by majorization on its subset.

\begin{thm*}
\label{Dolzhenko} 
If $E$ is a measurable subset of $\mathbb{T}$
or of a line in the complex plane and $f\in\mathcal{R}_{n}$, $|f(u)|\leq1$, 
$u\in E$,
then $\int_{E}|f'(u)|\,|du|\leq 2\pi n$; 
if $f$ is real valued, then the constant $2\pi$ can be replaced by $2$. 
The latter estimate is sharp for $n=0,1,2,\dots$ 
and for each $E$ with positive measure. 
\end{thm*}

We refer to a recent paper by N.K. Nikolski \cite{NN}
for a detailed account of the history of inequality \eqref{dolz}
and its relations to the Kreiss Matrix Theorem, as well as
for some new developments in the Kreiss theory.

Let us also mention that the Dolzhenko--Spijker inequality \eqref{dolz}
follows easily from (\ref{BE}) (or from (\ref{LR}).
This was mentioned by X. Li \cite{Li} (see also \cite[Remark 7]{FLM}). 
Indeed, integrating (\ref{BE}) with respect
to the normalized Lebesgue measure ${m}$, we obtain \eqref{dolz}.
\medskip


\subsubsection{S.M. Nikolskii's inequalities}

By the well-known results of S.M. Nikolskii \cite{SMNik} (1951), 
the essentially sharp inequality 
\begin{equation}
\label{Nik2}
\left\Vert f\right\Vert _{L^{q}}\leq c(p,\, q)
n^{\frac{1}{p}-\frac{1}{q}}\left\Vert f\right\Vert _{L^{p}}
\end{equation}
holds for all polynomials $f$ of degree at most $n$
and for all $1\leq p<q\leq\infty $  
with a constant $c(p,\, q)$ depending only on $p$ and $q$.
Analogous inequalities were proved in
\cite{SMNik} for trigonometric polynomials of several variables.
A few years later (1954) G. Szeg\"o and A. Zygmund \cite{SzZy}
among other more general results rediscovered (\ref{Nik2})
(for polynomials of one variable only) and extended it 
to the whole range $0<p<q\leq\infty $. 
\bigskip


\section{\label{Sec:Main-results}Main results}

In the present paper we use an approach based on the model space theory
to obtain some extensions of inequalities of type (\ref{BE}), (\ref{dolz}),
and (\ref{Nik2}) for functions in $\mathcal{R}_{n}$ or $\mathcal{R}_{n,\, r}$. 
The estimates we obtain will depend not only on the degree
of a rational function $n$ but also on the distance $1-r$ from 
the poles to the boundary. We show the sharpness of the obtained inequalities as both
$n\to \infty$ and $r\to 1-$.

Our main tools are integral representations for rational functions and their 
derivatives. Integral representations for a derivative 
of a rational function were introduced 
and successfully used by R.~Jones, X.~Li, R.N.~Mohapatra and R.S.~Rodriguez 
\cite[Lemma 4.3]{JLMR} (for rational functions without poles in $\overline{\mathbb{D}}$) 
and by X.~Li \cite[Lemma 3]{Li} in the general case. At the same time
such representations are known in the setting
of general model spaces where they were also used to produce some
estimates for the derivatives \cite{Aleks,Bar,dy}. 

From now on, for two positive functions $U$ and $V$, we say that $U$ is dominated
by $V$, denoted by $U\lesssim V$, if there is a constant $c>0$
such that $U\leq cV$; we say that $U$ and $V$ are comparable,
denoted by $U\asymp V$, if both $U\lesssim V$ and $V\lesssim U$.


\subsection{$L^{p}$-version of the Borwein--Erd\'elyi inequality}

We first give an analog of \eqref{LR}--\eqref{BE} in which the $L^{\infty}$-norm
at the right is replaced by the $L^{p}$-norm, $p\geq1.$

\begin{thm}
\label{TH_BE_Gen} 
Let $1\leq p\leq\infty $. Then

$({\rm i})$ For any function $f\in\mathcal{R}_{n}$ 
with the set of poles $a = \{a_{k}\}$ 
\textup(repeated counting multiplicities\textup) we have 
\begin{equation}
|f'(\xi)|\leq 
\bigg(\mathcal{D}^{\frac{1}{p}}_1(a)
\sum_{\left|a_{k}\right|<1}
\frac{1-\left|a_{k}\right|^{2}}{\left|a_{k}-\xi\right|^{2}}
+
\mathcal{D}^{\frac{1}{p}}_2(a)
\sum_{\left|a_{k}\right|>1}\frac{\left|a_{k}
\right|^{2}-1}{\left|a_{k}-\xi\right|^{2}} \bigg)\big\|f\big\|_{L^{p}} 
\label{BEp}
\end{equation}
for all $|\xi|=1$, where 
\begin{equation}
\label{dd}
\mathcal{D}_1(a) = 
\sum_{\left|a_{k}\right|<1}\frac{1+\left|a_{k}\right|}{1-\left|a_{k}\right|}, \qquad
\mathcal{D}_2(a) =
\sum_{\left|a_{k}\right|>1}
\frac{\left|a_{k}\right|+1}{\left|a_{k}\right|-1}.
\end{equation}

$({\rm ii})$ Moreover, \eqref{BEp} is sharp in the following sense:
for any $p\in [1, \infty]$ there exists a constant $c(p)>0$ such that 
for any $n\ge 2$ and any $r\in (0,1)$ there exists 
$f\in\mathcal{R}_{n}$ with the poles $\{a_{k}\}$ on the circle $\{|z| = \frac{1}{r}\}$ 
such that 
\[
\frac{f'(-1)}{\big\| f\big\|_{L^{p}}} 
\ge c(p) \mathcal{D}^{\frac{1}{p}}_2(a)
\sum_{\left|a_{k}\right|>1}\frac{\left|a_{k}\right|^{2}-1}{\left|a_{k}+1\right|^{2}}.
\]
\end{thm}

Clearly, inequality \eqref{LR} is the limit case of \eqref{BEp} when $p=\infty$.
\medskip


\subsection{An $L^{p}$-version of the Dolzhenko--Spijker Lemma}

Next we obtain a version of \eqref{dolz} where the $L^{\infty}$-norm
at the right is replaced by the $L^{p}$-norm, $p \geq 1$. 
Note that this inequality does not follow from (\ref{BEp}) 
by integration on the unit circle $\mathbb{T}$,
as it was the case for $p=1$. 

\begin{thm}
\label{ThGenSpLem} 
For every rational function $f \in \mathcal{R}_{n}$ 
having $n_1$ poles inside $\mathbb{D}$ and $n_2$ poles outside of  
$\mathbb{\overline{D}}$, we have
\begin{equation}
\label{GenSpLem}
\left\Vert f'\right\Vert _{L^{1}}\leq 
\Big(n_1^{1-\frac{1}{p}} \mathcal{D}^{\frac{1}{p}}_1(a)+
n_2^{1-\frac{1}{p}}\mathcal{D}^{\frac{1}{p}}_2(a)\Big)
\big\| f\big\|_{L^{p}}
\end{equation}
where $a=\{a_{k}\}$ stand for the poles of $f$ and $\mathcal{D}_1(a)$
and $\mathcal{D}_2(a)$ are defined in \eqref{dd}.
\end{thm}

This inequality is asymptotically sharp when $n = n_1+n_2$ tend to $\infty$
and ${\rm dist}\, (\{a_k\}, \mathbb{T}) \to 0$ (see Theorem \ref{ThBernLpLq} below).
\medskip


\subsection{$L^{p}-L^{q}$ Bernstein-type inequality}

Now we consider the following Bernstein-type problem, which
could be interpreted as a generalization of 
\eqref{BEp} and \eqref{GenSpLem}:
given $n\geq 1$, $r\in(0,\,1)$ and $1\leq p,\, q\leq\infty $,
let $\mathcal{C}_{n,\, r}\left(L^{q},\, L^{p}\right)$ be the best possible constant
in the inequality
\[
\left\Vert f'\right\Vert _{L^{q}}\leq\mathcal{C}_{n,\, r}\left(L^{q},\, 
L^{p}\right)\left\Vert f\right\Vert _{L^{p}},\qquad f\in\mathcal{R}_{n,\, r}.
\]
One could also introduce $\mathcal{C}_{n}\left(L^{q},\, L^{p}\right)=
\sup_{r\in(0,1)}\mathcal{C}_{n,\, r}\left(L^{q},\, L^{p}\right).$
The Dolzhenko--Spijker Lemma means exactly that  
$\mathcal{C}_{n}\left(L^{1},\, L^{\infty}\right)=n$.
It is however easy to see (take $f(z)=(1-rz)^{-1}$ as a test function)
that $\mathcal{C}_{n}\left(L^{q},\, L^{p}\right) = \infty$
unless $q=1$ and $p=\infty$. Thus, the dependence on $r$
(that is, on the distance from the poles to the boundary)
appears naturally in the problem.

\begin{thm}
\label{ThBernLpLq}
Let $n\geq1,$ $r\in(0,\,1),$ and $1\leq p,\, q\leq\infty .$
We have 
\begin{equation}
\label{eq:qgeqp}
\mathcal{C}_{n,\, r}\left(L^{q},\, L^{p}\right) \asymp
\begin{cases}
\left(\frac{n}{1-r}\right)^{1+\frac{1}{p}-\frac{1}{q}}, & q\geq p, \\
\frac{n}{\left(1-r\right)^{1+\frac{1}{p}-\frac{1}{q}}},      & q\leq p,
\end{cases}
\end{equation}
with the constants depending only on $p$ and $q$, but not on $n$ and $r$.

Moreover, the constant in upper bound 
is in both of these two cases $\left(1+r\right)^{1+\frac{1}{p}-\frac{1}{q}}$: we have
\begin{equation}
\label{qgeqp_up}
\mathcal{C}_{n,\, r}\left(L^{q},\, L^{p}\right)\leq  
\begin{cases}
(1+r)^{1+\frac{1}{p}-\frac{1}{q}}\left(\frac{n}{1-r}\right)^{1+\frac{1}{p}-\frac{1}{q}}, & q\geq p, \\
(1+r)^{1+\frac{1}{p}-\frac{1}{q}}\frac{n}{\left(1-r\right)^{1+\frac{1}{p}-\frac{1}{q}}}, & q\leq p. 
\end{cases}
\end{equation}
\end{thm}

The upper bound for $\mathcal{C}_{n,\, r}\left(L^{q},\, L^{p}\right)$
is obviously sharp for the special case $p=q=\infty $, since it is
reached by the Blaschke product $b_{r}^{n}$. Moreover, it is proved
in \cite{Z1} that for any fixed $r$ in (0, 1), there exists a limit
\[
\lim_{n\rightarrow\infty}\frac{\mathcal{C}_{n,\, r}\left(L^{2},\, L^{2}\right)}{n}=\frac{1+r}{1-r},
\]
and thus, the bound $(1+r)^{1+\frac{1}{p}-\frac{1}{q}}$
is again (asymptotically as $n\rightarrow\infty$) sharp for 
$p=q=2.$ It is possible that this constant is sharp in
the general case $1\leq p,\, q\leq\infty$.

In Subsection \ref{ooo} we compare inequality \eqref{qgeqp_up}
with a theorem by K.M. Dyakonov \cite[Theorem 11]{Dya3}.
\medskip


\subsection{An extension of S.M. Nikolskii's inequality to rational functions}
Direct analogs (with the dependence on $n$ only) of (\ref{Nik2}) do not exist
for functions in $\mathcal{R}_{n}$. As always, it is easy to check this
fact considering the test function $f(z)  = (1-rz)^{-1}$ as $r$ tends
to $1^{-}$. A natural extension of (\ref{Nik2}) for functions in 
$\mathcal{R}_{n,\,r}$ can be stated as follows: 

\begin{thm}
\label{ThNikGen} 
Let $1\leq p<q\leq\infty$, $n\geq1$ and $r\in(0,\,1).$

$({\rm i})$ We have  
\begin{equation}
\left\Vert f\right\Vert _{L^{q}}\leq
\left(\frac{1+r}{1-r}\right)^{\frac{1}{p}-\frac{1}{q}}
\left(n_1^{\frac{1}{p}-\frac{1}{q}}+(n_2+1)^{\frac{1}{p}-\frac{1}{q}}\right)
\left\Vert f\right\Vert _{L^{p}},\:\:\qquad 
f\in\mathcal{R}_{n,\, r},\label{inequ1}
\end{equation}
where $n_1$ \textup(respectively, $n_2$\textup) 
is the number of poles of $f$ inside $\mathbb{D}$
\textup(respectively, outside $\overline{\mathbb{D}}$\textup). In particular, 
\begin{equation}
\label{inequ2}
\left\Vert f\right\Vert _{L^{q}}\lesssim\left(\frac{n}{1-r}\right)^{\frac{1}{p}-\frac{1}{q}}\left\Vert f\right\Vert _{L^{p}},\:\:\qquad 
f\in\mathcal{R}_{n,\, r},
\end{equation}
with a constant depending only on $p$ and $q$.

$({\rm ii})$ The inequality \textup{(\ref{inequ2})} is sharp: 
for $1\le p<q\leq\infty$ there exists a constant $c(p,q)>0$
such that for any $r\in(0,\,1)$ and $n\geq2$ 
there exists $f\in\mathcal{R}_{n,\, r}$ with the property
\[
\frac{\left\Vert f\right\Vert _{L^{q}}}{\left\Vert f\right\Vert _{L^{p}}}
\ge c(p,q) \left(\frac{n}{1-r}\right)^{\frac{1}{p}-\frac{1}{q}}.
\]
\end{thm}

\begin{rem*}
It is possible that for any $1\leq p<q\leq\infty$
the upper bound $(1+r)^{\frac{1}{p}-\frac{1}{q}}$ in (\ref{inequ1}) 
is asymptotically sharp as $n$ tends to infinity, for any fixed $r\in(0,\,1)$. 
We are able to provide a simple proof of this fact for the special case
$q=\infty,$ $2\leq p<\infty$. Indeed, using the test function
$f=\frac{1}{1+rz}\sum_{k=0}^{n-1}b_{-r}^{k}$, $r\in(0,\,1),$ we
clearly have $\left\Vert f\right\Vert _{L^{2}}^2=\frac{n}{1-r^{2}}$,
since the family $\left\{\frac{1}{1+rz}b_{-r}^{k}\right\}$ is orthogonal
in $L^{2}$. Moreover, $\left\Vert f\right\Vert _{L^{\infty}}=f(-1)=\frac{n}{1-r}$
and thus 
\[
\left\Vert f\right\Vert _{L^{p}}^{p} \leq  
\left\Vert f\right\Vert _{L^{\infty}}^{p-2}\left\Vert f\right\Vert _{L^{2}}^{2}
= \frac{1}{1+r}\left(\frac{n}{1-r}\right)^{p-1},\qquad 2\leq p<\infty .
\]
As a consequence, 
\[
\frac{\left\Vert f\right\Vert _{L^{\infty}}}{\left\Vert f\right\Vert _{L^{p}}}
 \geq \left(n \frac{1+r}{1-r}\right)^{\frac{1}{p}}, \qquad 2\leq p<\infty ,
\]
which gives the result since the poles of $f$ are all outside 
$\frac{1}{r}\mathbb{D}$, $n_2=n=\deg f$. 
\end{rem*}
\medskip


\section{\label{Sec:Integral-representations-for}Integral representations
for rational functions and their derivatives}

\subsection{\label{Subs_assump}Preliminaries} 
In what follows we may assume without loss of generality 
that the functions $f=\frac{P}{Q}\in\mathcal{R}_{n}$ we consider
are such that: 
\begin{enumerate}
 \item $\max\left(\deg\, P,\,\deg\, Q\right)=n,$ (otherwise $f\in\mathcal{R}_{m},\, m<n$), 
 \item $\deg\, P\leq\deg\, Q$ : indeed, if $n=\deg\, P>l=\deg\, Q$ and
$\xi\in\mathbb{T},$ then we define $g\left(\xi\right)=f\left(\overline{\xi}\right)$
and supposing that $P(z)=\sum_{k=0}^{n}a_{k}z^{k},$ $a_{n}\neq0$
and $Q(z)=\sum_{k=0}^{l}b_{k}z^{k},$ $b_{l}\neq0,$ we obtain (multiplying
by $\xi^{n}$) 
$$
g\left(\xi\right)  =\frac{\sum_{k=0}^{n}a_{k}
\xi^{-k}}{\sum_{k=0}^{l}b_{k}\xi^{-k}}
=  \frac{\sum_{k=0}^{n}a_{k}\xi^{n-k}}{\sum_{k=0}^{l}b_{k}\xi^{n-k}}=\frac{\sum_{j=0}^{n}a_{n-j}\xi^{j}}{\sum_{j=n-l}^{n}b_{n-j}\xi^{j}}=\frac{\widetilde{P}\left(\xi\right)}{\widetilde{Q}\left(\xi\right)},
$$
where $\widetilde{P}(z)=\sum_{j=0}^{n}a_{n-j}z^{j}$ and 
$\widetilde{Q}(z)=\sum_{j=n-l}^{n}b_{n-j}z^{j}$
are such that $\widetilde{P},\,\widetilde{Q}\in\mathcal{P}_{n}$ and
$\deg\,\widetilde{P}\leq\deg\,\widetilde{Q}.$ Moreover, we clearly
have $\left|g'\left(\xi\right)\right|=
\left|f'\left(\overline{\xi}\right)\right|$ for all $\xi\in\mathbb{T}$,
$\left\Vert g\right\Vert _{L^{p}}=\left\Vert f\right\Vert _{L^{p}}$, $1\leq p\leq\infty$,
and $\left\Vert g'\right\Vert _{L^{q}}=\left\Vert f'\right\Vert _{L^{q}}$,
$1\leq q\leq\infty$,
since $g'\left(\xi\right)=-\frac{1}{\xi^{2}}f'\left(\overline{\xi}\right),\;\xi\in\mathbb{T},$
and
 \item all the poles of $f$ (i.e., the zeros of $Q$) are pairwise distinct:
indeed, we can assume this perturbing slightly the poles of $f$ and
the result will follow by continuity. 
\end{enumerate}

From now on, for every function $f\in\mathcal{R}_{n}$ we will denote
by $\sigma_1$ and $\sigma_2$
the sets of poles of $f$ 
(repeated counting multiplicities) which are respectively inside $\mathbb{D}$ or 
outside $\overline{\mathbb{D}}$,
\[
\sigma_1 = \left\{ \overline{\lambda_{1}},\,\overline{\lambda_{2}},\,
\dots,\,\overline{\lambda_{n_1}}\right\}, 
\qquad 
\sigma_2 = \left\{ 1/\overline{\mu_{1}},\,1/\overline{\mu_{2}},\,\dots,
\,1/\overline{\mu_{n_2}}\right\}. 
\]
Also, we will denote by 
\[
B_{1}=\prod_{j=1}^{n_1}b_{\lambda_{j}}, \quad\widetilde{B}_{1}=
\prod_{j=1}^{n_1}b_{\overline{\lambda_{j}}}\quad \mbox{and}\quad 
B_{2}=\prod_{j=1}^{n_2}b_{\mu_{j}},
\]
the corresponding finite Blaschke products and by $k_{\xi}^{B_{1}},$
$k_{\xi}^{\widetilde{B}_{1}}$, $k_{\xi}^{B_{2}}$ the reproducing
kernels at the point $\xi$ of the corresponding model spaces. 
Under the assumption (3), $f$ can be written as  
\begin{equation}
\label{simplform}
f(\xi)=a+\sum_{k=1}^{n_1}\frac{c_{k}}{\xi-\overline{\lambda_{k}}}+
\sum_{k=1}^{n_2}\frac{d_{k}}{1-\overline{\mu_{k}}\xi},
\qquad c_{k},\, d_{k}\in\mathbb{C},
\end{equation}
where $n=n_1+n_2 =\deg\, f$ ($a=0$ if and only if $\deg\, P<\deg\, Q)$. We put 
$g(\xi)=\sum_{k=1}^{n_1}\frac{c_{k}}{\xi-\overline{\lambda_{k}}}$ and
$h(\xi)=\sum_{k=1}^{n_2}\frac{d_{k}}{1-\overline{\mu_{k}}\xi},$ so
that $f=a+g+h.$ We denote by $V_{\sigma_{1},\,\sigma_{2}}$ the vector
space of all functions of the form \eqref{simplform}.
\medskip


\subsection{\label{Subs:Int-rep-for-funct}Integral representation for rational
functions on the unit circle}

We first obtain an integral representation for a function $f\in\mathcal{R}_{n}$. 

\begin{lem}
\label{lemma1} Keeping the notations and assumptions 1 and 2 of Subsection
\ref{Subs_assump}, for any function $f\in\mathcal{R}_{n}$ we have
\begin{equation}
f(\xi)=\left\langle f,\,\phi_{\xi}\right\rangle ,\qquad|\xi|=1,\label{integrep}
\end{equation}
where $\phi_{\xi}(u)=k_{\xi}^{zB_{2}}(u)+\xi\overline{uk_{\xi}^{\widetilde{B_{1}}}(u)}$,
$u,\,\xi\in\mathbb{T}$,
and $\left\langle \cdot,\,\cdot\right\rangle =\left\langle \cdot,\,\cdot\right\rangle _{L^{2}}$
stands for the scalar product in $L^{2}=L^{2}(\mathbb{T}, m).
$\end{lem}
\begin{proof}
Without loss of generality, we can suppose that $f$ satisfies assumption
3 of Subsection \ref{Subs_assump} for the same reason of continuity.
Thus, $f$ satisfies the above formula (\ref{simplform}): 
\begin{equation}
f(\xi)=a+h(\xi)+g(\xi).\label{decomp}
\end{equation}
Clearly, $a+h\in K_{zB_{2}}$. Thus, for a fixed $\xi$ we have 
\[
a+h(\xi)=\big<a+h,\, k_{\xi}^{zB_{2}}\big>.
\]
Moreover, since $g\in\overline{H_{0}^{2}}$ (where 
$H_{0}^{2}$ stands for the subspace
of $H^{2}$ consisting of functions $f$ such that $f(0)=0$), $ $we
have 
\begin{equation}
a+h(\xi)=\big<a+g+h,\, k_{\xi}^{zB_{2}}\big>=\big<f,\, k_{\xi}^{zB_{2}}\big>,\qquad|\xi|\leq1,\label{reph}
\end{equation}
Note that by the continuity of the kernel $k_{\xi}^{zB_{2}}$
in $\overline{\mathbb{D}}\times\overline{\mathbb{D}}$, this formula
extends to $\xi\in\mathbb{T}$.

To obtain an analogous formula for $g$, consider the function 
\begin{equation}
\label{phi}
\varphi(\xi)=\frac{1}{\xi}g\left(\frac{1}{\xi}\right),
\end{equation}
which belongs to $K_{B_{1}}$ and as a consequence we can write,
for $|\xi|<1$, 
\begin{equation}
\label{vseo}
\frac{1}{\xi}g\left(\frac{1}{\xi}\right)=\big<\varphi,\, k_{\xi}^{B_{1}}\big>=\int_{\mathbb{T}}\varphi(u)\frac{1-\overline{B_{1}(u)}B_{1}(\xi)}{1-\overline{u}\xi}{\rm d}{m}(u).
\end{equation}
Now setting $w=\frac{1}{\xi},$ $|w|>1$, changing the variable $v=\bar{u}$
and using the fact that $\varphi(u)=\bar{u}g(\bar{u})=vg(v)$, $u\in\mathbb{T}$,
we get 
\begin{equation}
\label{repg1}
g(w) = \int_{\mathbb{T}}\frac{1}{u}g\left(\frac{1}{u}\right)\frac{1-\overline{B_{1}(u)}B_{1}\left(\frac{1}{w}\right)}{w-\overline{u}}{\rm d}{m}(u)
= \int_{\mathbb{T}}g(v)v\frac{1-\overline{B_{1}(\overline{v})}B_{1}\left(\frac{1}{w}\right)}{w-v}{\rm d}{m}(v).
\end{equation}
Note also that (\ref{repg1}) holds for $|w|>1$ and, by continuity,
also for $|w|=1$. Now, for $|w|=1$,
\begin{equation}
\label{repg}
\begin{aligned}
g(w) & = \overline{w} \int_{\mathbb{T}}g(v)v\frac{1-\widetilde{B_{1}}(v)\overline{\widetilde{B_{1}}(w)}}{1-v\overline{w}}{\rm d}{m}(v) \\
     & = \overline{w} \big<g,\,\overline{zk_{w}^{\widetilde{B_{1}}}}\big> = 
         \big<f,\, w\overline{zk_{w}^{\widetilde{B_{1}}}}\big>,
\end{aligned}
\end{equation}
where the finite Blaschke product $\widetilde{B_{1}}$ is defined
in Subsection \ref{Subs_assump} (we have $B_{1}(\overline{v})=\overline{\widetilde{B_{1}}(v)}$,
$|v|=1$), and the last equality is due to the fact that the function
$\overline{zk_{w}^{\widetilde{B_{1}}}}$ belongs to $\overline{H_{0}^{2}}$
and thus, is orthogonal to $a+h.$ Now, combining (\ref{reph}) and
(\ref{repg}), we obtain for any $\xi\in\mathbb{T}:$ 
\[
f(\xi)  = a+ h(\xi)+g(\xi)  =  
 \big<f,\, k_{\xi}^{zB_{2}}+\xi\overline{zk_{\xi}^{\widetilde{B_{1}}}}\big>,
\]
which completes the proof. 
\end{proof}


\subsection{\label{Subs:Int-rep-for-deriv-of-funct}Integral representation for
the derivative of rational functions on the unit circle}

The integral representation in this subsection
essentially coincides with the representation due to X. Li
\cite[Lemma 3]{Li} (for the case of rational functions
without poles in $\overline{\mathbb{D}}$
it was proved by R.~Jones, X.~Li, R.N.~Mohapatra and 
R.S.~Rodriguez \cite[Lemma 4.3]{JLMR}).
It is possible to reinterpret the proof of Li's Lemma 
using the theory of model spaces. 
 
\begin{lem}
\label{lemma2} Keeping the notations and assumptions 1 and 2 of Subsection
\ref{Subs_assump}, for any function $f\in\mathcal{R}_{n}$, 
\begin{equation}
f'(\xi)=\left\langle f,\,\psi_{\xi}\right\rangle, \qquad |\xi|=1,\label{integrepder}
\end{equation}
where $\psi_{\xi}(u)=u\left(k_{\xi}^{B_{2}}(u)\right)^{2}-
\xi^{2}\overline{u}\overline{\left(k_{\xi}^{\widetilde{B}_{1}}(u)\right)^{2}}$,
$u,\,\xi\in\mathbb{T}$. 
\end{lem}

\begin{proof}
The scheme of the proof repeats the one of Lemma \ref{lemma1}. We
use again the fact that $f$ can be written as in (\ref{simplform}).
This time, we notice first that $h\in K_{B_{2}}$. Then for a fixed
$\xi$, we have 
\[
h'(\xi)=\bigg<h,\,\frac{z}{(1-\overline{\xi}z)^{2}}\bigg>=\big<h,\, z\left(k_{\xi}^{B_{2}}\right)^{2}\big>.
\]
Here the first equality is the standard Cauchy formula, while the
second follows from the fact that $z(1-\overline{\xi}z)^{-2}-
z\left(k_{\xi}^{B_{2}}(z)\right)^{2} \in B_{2}H^{2}$ and $h\perp B_{2}H^{2}$.
Thus, for the case when all poles are outside the disc
$\overline{\mathbb{D}}$, the formula is immediate.
Moreover, since $g\in\overline{H_{0}^{2}}$ (where $\overline{H_{0}^{2}}$
is defined above in the proof of Lemma \ref{lemma1}) we have 
\begin{equation}
h'(\xi)=\big<g+h,\, z\left(k_{\xi}^{B_{2}}\right)^{2}\big>=\big<f,\, z\left(k_{\xi}^{B_{2}}\right)^{2}\big>,\qquad|\xi|\leq1,\label{1}
\end{equation}
Again, by the continuity of the kernel 
in $\overline{\mathbb{D}}\times\overline{\mathbb{D}}$, this formula
extends to $\xi\in\mathbb{T}$.

To obtain an analogous formula for $g'$, consider the function $\varphi \in K_{B_{1}}$
defined by (\ref{phi}). Now setting in \eqref{vseo} 
$w=\frac{1}{\xi}$, $|w|>1$, we get 
\[
g(w)=\int_{\mathbb{T}}\varphi(u)\frac{1-\overline{B_{1}(u)}B_{1}\left(\frac{1}{w}\right)}{w-\overline{u}}{\rm d}{m}(u).
\]
Then, differentiating with respect to $w$, we obtain 
\[
g'(w)=\int_{\mathbb{T}}\varphi(u)\bigg(\frac{1-\overline{B_{1}(u)}B_{1}\left(\frac{1}{w}\right)}{w-\overline{u}}\bigg)'_{w}{\rm d}{m}(u)=-\int_{\mathbb{T}}\varphi(u)\bigg(\frac{1-\overline{B_{1}(u)}B_{1}\left(\frac{1}{w}\right)}{w-\overline{u}}\bigg)^{2}{\rm d}{m}(u),
\]
since, by a direct computation, 
\[
\bigg(\frac{1-\overline{B_{1}(u)}B_{1}\left(\frac{1}{w}\right)}{w-\overline{u}}\bigg)'_{w}+\bigg(\frac{1-\overline{B_{1}(u)}B_{1}\left(\frac{1}{w}\right)}{w-\overline{u}}\bigg)^{2}\in B_{1}H^{2}
\]
(as a function of $u$), while $\varphi\perp B_{1}H^{2}$. Changing
the variable $v=\bar{u}$ and using that $\varphi(u)=\bar{u}g(\bar{u})=vg(v)$,
$u\in\mathbb{T}$, we get 
\[
g'(w)=-\int_{\mathbb{T}}g(v)v\bigg(\frac{1-\overline{B_{1}(\bar{v})}B_{1}\left(\frac{1}{w}\right)}{w-v}\bigg)^{2}{\rm d}{m}(v).
\]

Recalling that the finite Blaschke product $\widetilde{B_{1}}$ defined
in Subsection \ref{Subs_assump} satisfies $B_{1}(\overline{v})=\overline{\widetilde{B_{1}}(v)}$,
$|v|=1$, we have
\begin{equation}
\begin{aligned}g'(w) & =-\int_{\mathbb{T}}g(v)v\bigg(\frac{1-\widetilde{B_{1}}(v)\overline{\widetilde{B_{1}}(w)}}{w-v}\bigg)^{2}{\rm d}m(v)\\
 & =-\int_{\mathbb{T}}f(v)v\bigg(\frac{1-\widetilde{B_{1}}(v)\overline{\widetilde{B_{1}}(w)}}{w-v}\bigg)^{2}{\rm d}m(v).
\end{aligned}
\label{2}
\end{equation}
The last equality follows from the fact that $h\in H^{2}$, while
$v\Big(\frac{1-\widetilde{B_{1}}(v)\overline{\widetilde{B_{1}}(w)}}{w-v}\Big)^{2}\in H_{0}^{2}$.
Note also that (\ref{2}) holds for $|w|>1$ and, by continuity, also
for $|w|=1$.

Now, applying formulas (\ref{1})--(\ref{2}) to $z=w\in\mathbb{T}$
and recalling that $f=a+g+h,$ $a\in\mathbb{C}$, we conclude that:
\[
f'(\xi)  =\int_{\mathbb{T}}(h+g)(u) \overline{\psi_\zeta(u)}{\rm d}{m}(u)
         =\left\langle f,\,\psi_{\xi}\right\rangle,
\]
since $\psi_{\xi}$ is orthogonal to 1. 
\end{proof}

\begin{rem*}
The above integral representation immediately implies inequality 
(\ref{LR}) by Levin and Rusak. Indeed, 
$|f'(\xi)| \le \|\psi_\xi\|_{L^1} \|f\|_{L^\infty} 
= (|B_1'(\xi)| + |B_2'(\xi)|)\|f\|_{L^\infty}$
(see \eqref{fact} below). It is, however, unclear,
whether one can prove the Borwein--Erd\'elyi inequality \eqref{BE}
using the representation \eqref{integrepder}.
\end{rem*}

\bigskip


\section{Proofs of the upper bounds}

In this Section we prove the upper bounds from 
Section \ref{Sec:Main-results}, that is, inequalities (\ref{BEp}),
(\ref{GenSpLem}), (\ref{qgeqp_up}) and (\ref{inequ1}).
All these proofs are based on
\begin{enumerate}
  \item integral representations (\ref{integrep}) or (\ref{integrepder});
  \item estimates of $H^{p}$-norms of reproducing kernels $k_{\xi}^{B},$
        where $B$ is a finite Blaschke product.
\end{enumerate}
In particular, we will often use the fact that for any finite Blaschke
product $B = \prod_{j=1}^{d}b_{\nu_{j}}$, $\nu_{j}\in\mathbb{D}$,
and for any $\xi\in\mathbb{T},$ 
\begin{equation}
\label{fact}
\big\| k_{\xi}^{B}\big\|_{L^{2}}^{2}=\left|B'(\xi)\right| = 
\sum_{j=1}^d \frac{1-|\nu_j|^2}{|\xi-\nu_j|^2}.
\end{equation}
We will use here the assumptions and the notations of Subsection \ref{Subs_assump}.


\subsection{Proof of Theorem \ref{TH_BE_Gen}, inequality \eqref{BEp}}
We first assume $1\leq p<\infty $ and denote by $p'$ the conjugate
exponent for $p$. Applying the H\"older inequality to the identity 
(\ref{integrepder}), we obtain
\begin{equation}
\label{first_holder}
\begin{aligned}
|f'(\xi)| & \leq\|f\|_{L^{p}}\big\|\psi_{\xi}\big\|_{L^{p'}}\nonumber \\
 & \leq\|f\|_{L^{p}}\left(\big\|(k_{\xi}^{\widetilde{B}_{1}})^{2}\big\|_{L^{p'}}+\big\|(k_{\xi}^{B_{2}})^{2}\big\|_{L^{p'}}\right)\nonumber \\
 & =\|f\|_{L^{p}}\left(\big\| k_{\xi}^{\widetilde{B}_{1}}\big\|_{H^{2p'}}^{2}+\big\| k_{\xi}^{B_{2}}\big\|_{H^{2p'}}^{2}\right),\qquad|\xi|=1.
\end{aligned}
\end{equation}
Now for any finite Blaschke product $B=\prod_{j=1}^{d}b_{\nu_{j}}$, 
$\nu_{j}\in\mathbb{D}$, we have for $|\xi|=|u|=1$ 
\begin{equation}
\label{repr_kern_estim_Lp}
  \big\| k_{\xi}^{B}\big\|_{H^{2p'}}^{2p'}  =
  \int_{\mathbb{T}}\left|k_{\xi}^{B}(u)\right|^{2(p'-1)+2}{\rm d}{m}(u)
  \leq \big\| k_{\xi}^{B}\big\|_{H^{2}}^{2}\max_{|u|=1}\left|k_{\xi}^{B}(u)\right|^{2(p'-1)}.
\end{equation}
On one hand, by \eqref{fact},
\[
\big\| k_{\xi}^{B}\big\|_{H^{2}}^{2}=\sum_{k=1}^{d}\frac{1-\left|\nu_{k}\right|^{2}}{\left|1-\nu_{k}\xi\right|^{2}},\qquad|\xi|=1,
\]
and on the other hand that for any $u,\,\xi\in\mathbb{T},$ 
\begin{equation}
  \label{CS0}
  \vert k_{\xi}^{B}(u)\vert\leq\sum_{j=1}^{d}\frac{1-\left|\nu_{j}\right|^{2}}{\left|1-\overline{\nu_{j}}\xi\right|\left|1-\overline{\nu_{j}}u\right|},
\end{equation}
which, by the Cauchy--Schwarz inequality, gives 
\begin{equation}
 \label{CS1}
 \vert k_{\xi}^{B}(u)\vert^{2}\leq\sum_{j=1}^{d}\frac{1-\left|\nu_{j}
 \right|^{2}}{\left|1-\overline{\nu_{j}}\xi\right|^{2}}\sum_{j=1}^{d}\frac{1+\left|\nu_{j}\right|}{1-\left|\nu_{j}\right|}, 
 \qquad u\in\mathbb{T}.
\end{equation}
Thus combining (\ref{CS1}) with (\ref{repr_kern_estim_Lp}), we obtain:
\begin{equation}
 \label{repr_kern_estim_Lp_2}
 \big\| k_{\xi}^{B}\big\|_{H^{2p'}}^{2p'}  
 \leq\bigg(\sum_{k=1}^{d}\frac{1-\left|\nu_{k}\right|^{2}}{\left|1-\nu_{k}\xi\right|^{2}}\bigg)^{p'}\bigg(\sum_{j=1}^{d}\frac{1+\left|\nu_{j}\right|}{1-\left|\nu_{j}\right|}\bigg)^{p'-1}.
\end{equation}
Applying the last inequality to $\widetilde{B}_1$ and $B_2$ we get
\[ 
\big\| k_{\xi}^{\widetilde{B}_{1}}\big\|_{H^{2p'}}^{2}+
\big\| k_{\xi}^{B_{2}}\big\|_{H^{2p'}}^{2} 
\leq \sum_{k=1}^{n_1}\frac{1-\left|\lambda_{k}\right|^{2}}{\left|1-\lambda_{k}\xi\right|^{2}}
 \bigg(\sum_{k=1}^{n_1}\frac{1+\left|\lambda_{k}\right|}{1-\left|\lambda_{k}\right|}\bigg)^{\frac{1}{p}}
+ \sum_{k=1}^{n_2}\frac{1-\left|\mu_{k}\right|^{2}}{\left|1-\overline{\mu_{k}}\xi\right|^{2}}
 \bigg(\sum_{k=1}^{n_2}\frac{1+\left|\mu_{k}\right|}{1-\left|\mu_{k}\right|}\bigg)^{\frac{1}{p}}, 
\]
which, in view of the definition of 
$\mathcal{D}_1(a)$ and $\mathcal{D}_2(a)$, completes the proof. 
\qed
\medskip


\subsection{Proof of Theorem \ref{ThGenSpLem}}
We use the integral representation (\ref{integrepder}) and Fubini--Tonelli's
Theorem to get: 
\begin{equation}
\label{nrk11-1}
\begin{aligned}
\left\Vert f'\right\Vert _{L^{1}} 
 & \leq  \int_{\mathbb{T}}\bigg(\int_{\mathbb{T}}\left|f(\tau)\right|
\Big(\big|k_{\xi}^{B_{2}}(\tau)\big|^{2}+\big|k_{\xi}^{\widetilde{B}_{1}}(\tau)\big|^{2}
\Big){\rm d}{m}(\tau)\bigg){\rm d}{m}(\xi) \\
 & =  \int_{\mathbb{T}}\left|f(\tau)\right|\left(\big|B_{2}'(\tau)\big|
+\big|\widetilde{B}_{1}'(\tau)\big|\right){\rm d}{m}(\tau),
\end{aligned}
\end{equation}
where the last equality comes from (\ref{fact}). Now, for any finite
Blaschke product $B=B_{\sigma}$ of degree $d,$ corresponding to
a set $\sigma=\left\{ \nu_{1},\,\nu_{2},\,\dots,\,\nu_{d}\right\} \subset r\mathbb{D}$,
we have $\left\Vert B'\right\Vert _{L^{1}}=d$ (this follows immediately if 
we integrate equality \eqref{fact}).

On the other hand, 
\begin{equation}
\label{LinftyderivBlaschke}
\left\Vert B'\right\Vert _{L^{\infty}} \leq 
\sum_{k=1}^{d}\bigg\Vert\frac{1-\left|\nu_{k}\right|^{2}}{\left(1-\overline{\nu_{k}}z\right)^{2}}\bigg\Vert_{L^{\infty}}
\leq  \sum_{k=1}^{d}\frac{1+\left|\nu_{k}\right|}{1-\left|\nu_{k}\right|}.
\end{equation}
This implies that for any $1\leq p'\leq\infty ,$ 
\begin{equation}
\label{LpderivBlaschke-1}
\left\Vert B'\right\Vert _{L^{p'}} \leq  
\left\Vert B'\right\Vert _{L^{\infty}}^{1-\frac{1}{p'}} 
\left\Vert B'\right\Vert _{L^{1}}^{\frac{1}{p'}} 
\leq  d^{\frac{1}{p'}}\bigg(\sum_{k=1}^{d}\frac{1+\left|\nu_{k}\right|}{1-\left|\nu_{k}\right|}\bigg)^{1-\frac{1}{p'}}.
\end{equation} 
Going back to (\ref{nrk11-1}) and applying the H\"older inequality, we get 
\begin{eqnarray*}
\left\Vert f'\right\Vert _{L^{1}} & = & 
\int_{\mathbb{T}}\left|f(\tau)\right|\left|B_{2}'(\tau)\right|{\rm d}{m}(\tau)+
\int_{\mathbb{T}}\left|f(\tau)\right| |\widetilde{B}_{1}'(\tau)| {\rm d}{m}(\tau)\\
 & \leq & \left\Vert f\right\Vert _{L^{p}}\left(\left\Vert B_{2}'\right\Vert _{L^{p'}}+\big\|\widetilde{B}_{1}'\big\|_{L^{p'}}\right),
\end{eqnarray*}
where $p'$ is the conjugate of $p,$. Applying the last inequality
combined with (\ref{LpderivBlaschke-1}) to $\widetilde{B}_1$ and $B_2$ we get
$$
\left\Vert f'\right\Vert _{L^{1}} \leq  \left\Vert f\right\Vert _{L^{p}}
\bigg(n_1^{1-\frac{1}{p}}\bigg(\sum_{k=1}^{n_1}\frac{1+\left|\lambda_{k}\right|}
{1-\left|\lambda_{k}\right|}\bigg)^{\frac{1}{p}}+n_2^{1-\frac{1}{p}}
\bigg(\sum_{k=1}^{n_2}\frac{1+\left|\mu_{k}\right|}{1-\left|\mu_{k}\right|}\bigg)^{\frac{1}{p}}\bigg),
$$
as required. 
\qed
\medskip


\subsection{Proof of Theorem \ref{ThBernLpLq}, inequality \eqref{qgeqp_up}}
\label{ooo}
The proof will consist of several steps.
\smallskip

\textbf{Step 1. The case $q=\infty$, $1\leq p\leq\infty$.} 
Clearly, for a function $f\in \mathcal{R}_{n,\, r}$ 
having $n_1$ poles inside $\mathbb{D}$
and $n_2$ poles outside $\overline{\mathbb{D}}$
we have
$$
\mathcal{D}_1 \le n_1\frac{1+r}{1-r}, \qquad 
\mathcal{D}_2 \le n_2\frac{1+r}{1-r}.
$$
Taking
the supremum over all $\xi\in\mathbb{T}$ in (\ref{BEp})
we obtain 
\begin{equation}
\label{Step1}
\big\| f'\big\|_{L^{\infty}} 
\le \left(\frac{1+r}{1-r}\right)^{1+\frac{1}{p}}(n_1+n_2) 
\big\| f\big\|_{L^{p}}.
\end{equation}
\smallskip

\textbf{Step 2. The case $q=1$, $1\leq p\leq\infty$.}
A direct consequence of the inequality (\ref{GenSpLem}) is that 
\begin{equation}
\big\| f'\big\|_{L^{1}}\leq\left(\frac{1+r}{1-r}\right)^{\frac{1}{p}}
(n_1 +n_2) \big\| f\big\|_{L^{p}},\label{Step2}
\end{equation}
for any $f\in\mathcal{R}_{n,\, r}$ having $n_1$ poles inside $\mathbb{D}$
and $n_2$ poles outside $\overline{\mathbb{D}}.$
\medskip

\textbf{Step 3. The case $p=q$.}
For any $f\in\mathcal{R}_{n,\, r}$ (as always of the form (\ref{simplform})), we have
by (\ref{integrepder})
\begin{eqnarray*}
\left\Vert f'\right\Vert _{L^{p}}^{p} & = & \int_{\mathbb{T}}\left|f'(\xi)\right|^{p}{\rm d}{m}(\xi)\\
 & = & \int_{\mathbb{T}}\bigg|\int_{\mathbb{T}}\overline{\tau}f(\tau)\overline{u\left(k_{\xi}^{B_{2}}\right)^{2}-\xi^{2}\overline{u}\overline{\left(k_{\xi}^{\widetilde{B}_{1}}\right)^{2}}}{\rm d}{m}(\tau)\bigg|^{p}{\rm d}{m}(\xi)\\
 & \leq & \int_{\mathbb{T}}\bigg(\int_{\mathbb{T}}\left|f(\tau)\right|\bigg(|k_{\xi}^{B_{2}}(\tau)|^{2}+|k_{\xi}^{\widetilde{B}_{1}}(\tau)|^{2}\bigg){\rm d}{m}(\tau)\bigg)^{p}{\rm d}{m}(\xi).
\end{eqnarray*}
Now applying the H\"older inequality ($p'$ being the conjugate of $p$), we obtain
\[
\begin{aligned}
\left\Vert f'\right\Vert _{L^{p}}^{p} & \le
\bigg(\int_{\mathbb{T}}\Big(\left|k_{\xi}^{B_{2}}(\tau)\right|^{2}+
|k_{\xi}^{\widetilde{B}_{1}}(\tau)|^{2}\Big){\rm d}{m}(\tau)\bigg)^{\frac{p}{p'}}
\int_{\mathbb{T}}\left|f(\tau)\right|^{p}\Big(
\left|k_{\xi}^{B_{2}}(\tau)\right|^{2}+ |k_{\xi}^{\widetilde{B}_{1}}(\tau)|^{2}\Big){\rm d}{m}(\tau) \\
& =\left(\left|B_{2}'(\xi)\right|+ |\widetilde{B}_{1}'(\xi)|\right)^{\frac{p}{p'}}
\int_{\mathbb{T}}\left|f(\tau)\right|^{p}\Big(\left|k_{\xi}^{B_{2}}(\tau)\right|^{2}+
|k_{\xi}^{\widetilde{B}_{1}}(\tau)|^{2}\Big){\rm d}{m}(\tau),
\end{aligned}
\]
where the last equality comes from (\ref{fact}). Now, integrating
the last inequality on the unit circle with respect to $\xi$, we
obtain 
\begin{eqnarray*}
\left\Vert f'\right\Vert _{L^{p}}^{p} & \leq & 
\int_{\mathbb{T}}\left(\left|B_{2}'(\xi)\right|+|\widetilde{B}_{1}'(\xi)|
\right)^{\frac{p}{p'}}\int_{\mathbb{T}}\left|f(\tau)\right|^{p}
\left(\left|k_{\xi}^{B_{2}}(\tau)\right|^{2}+|k_{\xi}^{\widetilde{B}_{1}}(\tau)|^{2}
\right){\rm d}{m}(\tau){\rm d}{m}(\xi)\\
 & \leq & \max_{|\xi|=1}\left(\left|B_{2}'(\xi)\right|+
|\widetilde{B}_{1}'(\xi)|\right)^{\frac{p}{p'}}
\int_{\mathbb{T}}\int_{\mathbb{T}}\left|f(\tau)\right|^{p}\left(
\left|k_{\xi}^{B_{2}}(\tau)\right|^{2}+|k_{\xi}^{\widetilde{B}_{1}}(\tau)|^{2}
\right){\rm d}{m}(\tau){\rm d}{m}(\xi)\\
 & = & \max_{|\xi|=1}\left(\left|B_{2}'(\xi)\right|+|\widetilde{B}_{1}'(\xi)|
\right)^{\frac{p}{p'}}\int_{\mathbb{T}}\left|f(\tau)\right|^{p}\left(
\left|B_{2}'(\tau)\right|+|\widetilde{B}_{1}'(\tau)|\right){\rm d}{m}(\tau).
\end{eqnarray*}
Finally, using (\ref{LinftyderivBlaschke}) we obtain 
\begin{eqnarray*}
\left\Vert f'\right\Vert _{L^{p}}^{p} & \leq & \max_{|\xi|=1}
\left(\left|B_{2}'(\xi)\right|+|\widetilde{B}_{1}'(\xi)|\right)^{\frac{p}{p'}+1}\left\Vert f\right\Vert _{L^{p}}^{p}\\
 & \leq & 
\bigg(\sum_{k=1}^{n_1}\frac{1+\left|\lambda_{k}\right|}{1-\left|\lambda_{k}\right|}
+ \sum_{k=1}^{n_2}\frac{1+\left|\mu_{k}\right|}{1-\left|\mu_{k}\right|}
\bigg)^{\frac{p}{p'}+1}\left\Vert f\right\Vert _{L^{p}}^{p},
\end{eqnarray*}
whence
\begin{equation}
\left\Vert f'\right\Vert_{L^{p}}\leq n 
\frac{1+r}{1-r}\left\Vert f\right\Vert _{L^{p}}.\label{Step3}
\end{equation}
\smallskip

\textbf{Step 4. The case $1 < q < p <\infty$.} 
This follows by interpolation between the cases $q=1$ and $q=p$.
For any $f\in\mathcal{R}_{n,\, r},$ we have, by the H\"older inequality with the exponents
$\frac{p-1}{p-q}$ and $\frac{p-1}{q-1}$,
$$
\left\Vert f'\right\Vert _{L^{q}}^{q}  
 =  \int_{\mathbb{T}}\left|f'(\tau)\right|^{\frac{p-q}{p-1}+\frac{p(q-1)}{p-1}}{\rm d}{m}(\tau)
\le  \left\Vert f'\right\Vert _{L^{1}}^{\frac{p-q}{p-1}}\left\Vert f'\right\Vert _{L^{p}}^{\frac{p(q-1)}{p-1}}.
$$
Now, using both the inequalities (\ref{Step2}) and (\ref{Step3})
from Steps 2 and 3, we obtain the required estimate:
$$
\begin{aligned}
\left\Vert f'\right\Vert _{L^{q}}^{q} &  
\le \bigg(n \Big(\frac{1+r}{1-r}\Big)^{\frac{1}{p}} 
\|f\|_{L^p}\bigg)^{\frac{p-q}{p-1}} 
\bigg(n \frac{1+r}{1-r}\|f\|_{L^p}\bigg)^{\frac{p(q-1)}{p-1}} \\
& \leq  n^q \left(\frac{1+r}{1-r}\right)^{\frac{1}{p}\frac{p-q}{p-1}+\frac{p(q-1)}{p-1}}\left\Vert f\right\Vert _{L^{p}}^{q}
  =  n^{q}\left(\frac{1+r}{1-r}\right)^{q-1+\frac{q}{p}}\left\Vert f\right\Vert _{L^{p}}^{q}.
\end{aligned}
$$
\smallskip

\textbf{Step 5. The case $1\leq p\leq q\leq\infty.$} 
We now interpolate between $q=p$ and $q=\infty$:
$$
\left\Vert f'\right\Vert _{L^{q}}^{q}  \leq 
\left\Vert f'\right\Vert _{L^{\infty}}^{q-p}\left\Vert f'\right\Vert _{L^{p}}^{p} 
 \le 
\left(n\frac{1+r}{1-r}\right)^{1+\frac{1}{p}-\frac{1}{q}}
\left\Vert f\right\Vert _{L^{p}},
$$
where we used inequalities (\ref{Step1}) and  (\ref{Step3}).
\qed

\begin{rem*}
For the case when $1\leq q\leq p\leq\infty$
and the function $f \in \mathcal{R}_{n}$ has no poles in $\overline{\mathbb{D}}$,
a different proof of the upper bound 
(but without explicit constants) could be given by an application
of a result by  K.M. Dyakonov \cite[Theorem 11]{Dya3}. 
Namely, applying inequality (11.2) from \cite{Dya3} with $s=1$
we obtain that for a rational function $f$ of degree $n$ 
with the poles $1/\overline{\nu_{1}},\,1/\overline{\nu_{2}},\,\dots,\,
1/\overline{\nu_{n}} \notin \overline{\mathbb{D}}$ we have
\begin{equation}
\left\Vert f'\right\Vert _{L^{q}}\leq C \left\Vert B'\right\Vert _{L^{\gamma}}\left\Vert f\right\Vert _{L^{p}},\label{Dyakinequ}
\end{equation}
where $q\leq p$, $\frac{1}{\gamma}=\frac{1}{q}-\frac{1}{p}$, 
$C$  is a constant depending on $q$ and $p$ which is not precised, and
$B=B_{\sigma}$ is the finite Blaschke product corresponding to 
$\sigma=\left(\nu_{1},\,\nu_{2},\,\dots,\,\nu_{n}\right)$.
It remains to apply inequality (\ref{LpderivBlaschke-1}).
\end{rem*}


\subsection{Proof of the upper bound in Theorem \ref{ThNikGen}}

\textbf{Step 1. The special case $q=\infty $ and $1\leq p\leq2.$}
Following the assumptions of Subsection \ref{Subs_assump} and applying
H\"older inequality to (\ref{integrep}), we obtain (with the notations
of Lemma \ref{lemma1}) that for any $\xi\in\mathbb{T},$ 
\begin{equation}
\left|f(\xi)\right|\leq\left\Vert f\right\Vert _{L^{p}}\left\Vert \phi_{\xi}\right\Vert _{L^{p'}},\label{hold1}
\end{equation}
where $p'\geq2$ stands for the conjugate of $p.$ Moreover, 
$$
\left\Vert \phi_{\xi}\right\Vert _{L^{p'}} 
 \le \big\Vert k_{\xi}^{zB_{2}}\big\Vert_{L^{p'}}+\big\Vert k_{\xi}^{\widetilde{B_{1}}}\big\Vert_{L^{p'}}.
$$
Now, we prove that for any finite Blaschke product $B=B_{\sigma}$
of degree $d,$ corresponding to a set $\sigma=\left\{ \nu_{1},\,\nu_{2},\,\dots,\,\nu_{d}\right\} \subset r\mathbb{D}$,
we have 
\begin{equation}
\left\Vert k_{\xi}^{B}\right\Vert _{L^{p'}}\leq\left(\frac{1+r}{1-r}d\right)^{1+\frac{1}{p}}.\label{eq:nrk2}
\end{equation}
Indeed, as a direct consequence of (\ref{CS0}), we get 
\begin{equation}
\label{infinitynorm}
\left\Vert k_{\xi}^{B}\right\Vert _{L^{\infty}}\leq\sum_{j=1}^{d}\frac{1+\left|\nu_{j}\right|}{1-\left|\nu_{j}\right|}\leq d \frac{1+r}{1-r}.
\end{equation}
Moreover, (\ref{fact}) clearly gives that for any $\xi\in\mathbb{T},$
\begin{equation}
\label{L2norm}
\left\Vert k_{\xi}^{B}\right\Vert _{L^{2}}^{2}=\sum_{j=1}^{d}\frac{1-\left|\nu_{j}\right|^{2}}{\left|1-\overline{\nu_{j}}\xi\right|^{2}}\leq\sum_{j=1}^{d}\frac{1+\left|\nu_{j}\right|}{1-\left|\nu_{j}\right|}\leq d \frac{1+r}{1-r}.
\end{equation}
Thus, combining (\ref{infinitynorm}) and (\ref{L2norm}), we get
that for any $p'\geq2$, 
$$
\left\Vert k_{\xi}^{B}\right\Vert _{L^{p'}}  \leq  
\left\Vert k_{\xi}^{B}\right\Vert _{L^{\infty}}^{1-\frac{2}{p'}}
\left\Vert k_{\xi}^{B}\right\Vert _{L^{2}}^{\frac{2}{p'}} 
\leq  d^{p'-1}\left(\frac{1+r}{1-r}\right)^{p'-1}.
$$
Applying this to $B=\widetilde{B}_1$ and to $B = zB_2$, we get 
\begin{equation}
\label{Lpkernelnorm1}
\left\Vert \phi_{\xi}\right\Vert _{L^{p'}}  \leq 
\left(\frac{1+r}{1-r}\right)^{1-\frac{1}{p'}}\big(n_1^{1-\frac{1}{p'}}
+(n_2+1)^{1-\frac{1}{p'}}\big),
\end{equation}
where $n_1$ (respectively, $n_2$) is the number of poles of $f$ inside $\mathbb{D}$
(respectively, outside $\overline{\mathbb{D}})$.
Thus, it follows from (\ref{hold1}) and (\ref{Lpkernelnorm1})
that for any $1\leq p\leq2,$ 
\begin{equation}
\label{Step1-1}
\left\Vert f\right\Vert _{L^{\infty}}
\leq\left(\frac{1+r}{1-r}\right)^{\frac{1}{p}}
\left(n_1^{\frac{1}{p}}+(n_2+1)^{\frac{1}{p}}\right)\left\Vert 
f\right\Vert _{L^{p}},
\end{equation}
as required.
\medskip

\textbf{Step 2. The case $1\leq p\leq2$ and $1\leq p<q\leq\infty$.}
For any $f\in\mathcal{R}_{n,\, r},$
\[
\left\Vert f\right\Vert _{L^{q}}^{q}\leq\left\Vert f\right\Vert _{L^{\infty}}^{q-p}\left\Vert f\right\Vert _{L^{p}}^{p},
\]
which gives using (\ref{Step1-1}), 
\begin{eqnarray*}
\left\Vert f\right\Vert _{L^{q}} & \leq & \left\Vert f\right\Vert _{L^{\infty}}^{1-\frac{p}{q}}\left\Vert f\right\Vert _{L^{p}}^{\frac{p}{q}}\nonumber \\
 & \leq & \left(\frac{1+r}{1-r}\right)^{\frac{1}{p}-\frac{1}{q}}\bigg(n_1^{\frac{1}{p}}+(n_2+1)^{\frac{1}{p}}\bigg)^{1-\frac{p}{q}}\left\Vert f\right\Vert _{L^{p}}\nonumber \\
 & \leq & \left(\frac{1+r}{1-r}\right)^{\frac{1}{p}-\frac{1}{q}}\left(n_1^{\frac{1}{p}-\frac{1}{q}}+
(n_2+1)^{\frac{1}{p}-\frac{1}{q}}\right)\left\Vert f\right\Vert _{L^{p}}. 
\end{eqnarray*}
\smallskip

\textbf{Step 3. The case $2\leq p\leq\infty $ and $1\leq p<q\leq\infty .$}
Let $p'$ and $q'$ be the conjugates to $p$ and $q$, respectively.
Then $1\le q' < p' \le 2$. By Step 2 we have 
for any $V_{\sigma_{1},\,\sigma_{2}}\subset\mathcal{R}_{n,\, r}$
(see Subsection \ref{Subs_assump} for the definition)
\[
\left\Vert Id\right\Vert _{\left(V_{\sigma_{1},\,\sigma_{2}},\, L^{q'}\right)
\rightarrow\left(V_{\sigma_{1},\,\sigma_{2}},\, L^{p'}\right)}
\leq\left(\frac{1+r}{1-r}\right)^{\frac{1}{q'}-\frac{1}{p'}}
\left(n_1^{\frac{1}{q'}-\frac{1}{p'}}+(n_2+1)^{\frac{1}{q'}-\frac{1}{p'}}\right),
\]
where $Id$ is the identity operator.
Now denoting by $Id^{\star}$ the adjoint operator (for the usual
Cauchy duality) of $Id$, we have $Id^{\star}=Id.$ 
Therefore,
\begin{eqnarray*}
\left\Vert Id\right\Vert _{\left(V_{\sigma_{1},\,\sigma_{2}},\, L^{p}\right)
  \rightarrow\left(V_{\sigma_{1},\,\sigma_{2}},\, L^{q}\right)} 
 & = & \left\Vert Id^{\star}\right\Vert _{\left(V_{\sigma_{1},\,\sigma_{2}},\, L^{p}\right)\rightarrow\left(V_{\sigma_{1},\,\sigma_{2}},\, L^{q}\right)}\nonumber \\
 & = & \left\Vert Id\right\Vert _{\left(V_{\sigma_{1},\,\sigma_{2}},\, L^{q'}\right)\rightarrow\left(V_{\sigma_{1},\,\sigma_{2}},\, L^{p'}\right)}\nonumber \\
 & \leq & \left(\frac{1+r}{1-r}\right)^{\frac{1}{q'}-\frac{1}{p'}}
\left(n_1^{\frac{1}{q'}-\frac{1}{p'}}+(n_2+1)^{\frac{1}{q'}-\frac{1}{p'}}\right)\nonumber \\
 & = & \left(\frac{1+r}{1-r}\right)^{\frac{1}{p}-\frac{1}{q}}
\left(n_1^{\frac{1}{p}-\frac{1}{q}}+(n_2+1)^{\frac{1}{p}-\frac{1}{q}}
\right),
\end{eqnarray*}
which is the required estimate. 
\qed
\bigskip


\section{\label{Subs_proof_sharpness}Proofs of the lower bounds}

In this section we prove the asymptotic sharpness of the inequalities 
in Theorems \ref{TH_BE_Gen}, \ref{ThBernLpLq} and \ref{ThNikGen} 
as $n$ tends to infinity and the poles of the
rational functions approach the unit circle $\mathbb{T}$. From now
on, we denote by 
\begin{equation}
D_{n}(z) = \sum_{k=0}^{n-1}z^{k} \label{def_Dir_kern}
\end{equation}
the Dirichlet kernel of order $n\geq 1$. 
The asymptotic behaviour of $\big\Vert D_{n}\big\Vert_{L^{q}}$ 
as $n$ tends to $\infty $ is well known: for $q>1$, 
\begin{equation}
\label{Fosc_result}
\lim_{n\rightarrow\infty }\frac{\big\Vert D_{n}\big\Vert_{L^{q}}}{n^{1-\frac{1}{q}}}=\bigg(\frac{1}{\pi}\int_{\mathbb{R}}\bigg\vert\frac{\sin x}{x}\bigg\vert^{q}{\rm d}x\bigg)^{\frac{1}{q}}.
\end{equation}

Put $I_{n}=\big\{ z\in\mathbb{T}:\: z=e^{it},\, t\in 
\big[\frac{\pi}{4n},\,\frac{\pi}{2n}\big]\big\}$.
It is well known and easy to see that for $q>1$ the integral over 
the arc $I_n$ gives a substantial contribution to  
the norms $\big\Vert D_{n}\big\Vert_{L^{q}}$ and $\big\Vert D_{n}D_{n}'\big\Vert_{L^{q}}$, 
namely
\begin{equation}
\label{Lq_norm_D_n}
\int_{I_{n}} \left|D_{n}(\xi)\right|^{q}{\rm d}{m}(\xi) 
\gtrsim n^{q-1}, \qquad
\int_{I_{2n}}\left|D_{n}(\xi)
D_{n}'(\xi)\right|^{q}{\rm d}{m}(\xi) \gtrsim n^{3q-1}.
\end{equation}

Now we introduce the test functions from $\mathcal{R}_{n}$
which will be used throughout this section 
to illustrate the sharpness of the inequalities
in Theorems \ref{TH_BE_Gen}, \ref{ThBernLpLq} and \ref{ThNikGen}. Put
\begin{equation}
\label{testfunc1}
f(z)=b_{-r}'(z)\sum_{k=0}^{n-2}b_{-r}^{k}(z), \qquad
b_{-r}(z) = \frac{z+r}{1+rz}.
\end{equation}
Thus, $f$ is essentially the Dirichlet
kernel transplanted from the origin to the point $\lambda=-r\in(0,\,1)$
by the change of variable $\circ b_{-r}$. Also, for $n\ge 4$, we put
\begin{equation}
\label{testfunc2}
g(z)=b_{-r}'(z)\bigg(\sum_{k=0}^{N}b_{-r}^{k}(z)\bigg)^{2}\in\mathcal{R}_{n},
\end{equation}
where $N$ is the integer part of $\frac{n-2}{2}$.
It is interesting to note that other natural (and simpler) candidates such 
as $h(z) = (1-rz)^{-n}$ or $h(z)=\frac{1}{1-rz} b_{r}^{n-1}(z)$ will not
give the right order of $n$ in the inequality
(\ref{qgeqp_up}) for $p<q$.

We will need the following lemma.

\begin{lem}
\label{bab}
We have
\begin{equation}
\label{Lp_norm_upper_estim_test_func_1}
\left\Vert f\right\Vert _{L^{p}}\lesssim
\left(\frac{n}{1-r}\right)^{1-\frac{1}{p}}, \qquad 1<p\leq\infty,
\end{equation}
and
\begin{equation}
\label{Lp_norm_test_func_2}
\left\Vert g\right\Vert _{L^{p}}\lesssim
\frac{n^{2-\frac{1}{p}}}{\left(1-r\right)^{1-\frac{1}{p}}}, \qquad 1 \le p\leq\infty.
\end{equation}
\end{lem}

\begin{proof}
Let us compute the $L^{p}$-norm of $f$. Making use of the change of variable 
$\zeta = b_{-r}(\xi)$ we get
\begin{align}
\left\Vert f\right\Vert _{L^{p}}^{p} & =\int_{\mathbb{T}}\left|b_{-r}'(\xi)\right|\left|b_{-r}'(\xi)\right|^{p-1}\left|\sum_{k=0}^{n-2}b_{-r}^{k}(\xi)\right|^{p}{\rm d}{m}(\xi)\nonumber \\
 & =\int_{\mathbb{T}}\left|b_{-r}'(b_{-r}(\zeta))\right|^{p-1}\left|D_{n-1}(\zeta)
\right|^{p}{\rm d}{m}(\zeta). 
\end{align}
By a straightforward computation, 
$b_{-r}'\circ b_{-r}=\frac{r^{2}-1}{(1+rb_{-r})^{2}} = -\frac{(1+rz)^{2}}{1-r^{2}}.$ Thus,
\begin{equation}
\label{Lp_first_estim_test_func_1}
\left\Vert f\right\Vert _{L^{p}}^{p} = 
\frac{1}{\left(1-r^{2}\right)^{p-1}}\int_{\mathbb{T}}\left|1+r\zeta\right|^{2(p-1)}\left|D_{n-1}(\zeta)\right|^{p}{\rm d}{m}(\zeta)\nonumber \\
 \leq \frac{(1+r)^{p-1}}{\left(1-r\right)^{p-1}}\left\Vert D_{n-1}\right\Vert _{L^{p}}^{p},
\end{equation}
and the statement follows from (\ref{Fosc_result}).

Analogously, changing the variable $\zeta = b_{-r}(\xi)$, we obtain
$$ 
\begin{aligned}
\left\Vert g\right\Vert _{L^{p}}^p 
 & = \int_{\mathbb{T}}\left|b_{-r}'(b_{-r}(\zeta))\right|^{p-1}\left|D_{N+1}(\zeta)\right|^{2p}{\rm d}{m}(\zeta) \\
 & = \frac{1}{\left(1-r^{2}\right)^{p-1}}\int_{\mathbb{T}}\left|1+r\zeta\right|^{2(p-1)}\left|D_{N+1}(\zeta)\right|^{2p}{\rm d}{m}(\zeta) \\
 & \leq \frac{(1+r)^{p-1}}{\left(1-r\right)^{p-1}}\left\Vert D_{N+1}\right\Vert _{L^{2p}}^{2p}\lesssim\frac{(1+r)^{p-1}}{\left(1-r\right)^{p-1}}n^{2p-1}. 
\end{aligned}
$$
\qed


\subsection{\label{subsub_sharp_BEp} Sharpness in Theorem \ref{TH_BE_Gen}}

We prove here the statement (ii) of Theorem \ref{TH_BE_Gen}. 
Without loss of generality we assume that $n\ge 5$ (for small $n$
the statement is obvious, take the test function $(1-rz)^{-2}$).
First we consider the case $1<p\leq\infty$. For $r\in(0,\,1)$ 
let $f$ be defined by \eqref{testfunc1}. We have 
\[
f'=b_{-r}''\sum_{k=0}^{n-2}b_{-r}^{k}+\left(b_{-r}'\right)^{2}\sum_{k=0}^{n-3}(k+1)b_{-r}^{k}.
\]
Moreover, $b_{-r}(-1)=1,$ $b_{-r}'(-1)=-\frac{1+r}{1-r},$
and $b_{-r}''(-1) =\frac{2r(1+r)}{(1-r)^{2}}.$
This gives 
$$
f'(-1)  =\frac{2r(n-1)(1+r)}{(1-r)^{2}}+\left(\frac{1+r}{1-r}\right)^{2}\sum_{k=0}^{n-3}(k+1)>\left(\frac{1+r}{1-r}\right)^{2}\sum_{k=0}^{n-3}(k+1)
\gtrsim\left(\frac{n}{1-r}\right)^{2}.
$$
On the other hand, for $a=\{a_k\}$, $a_k = \frac{1}{r}$, $k=1, \dots, n$,
we have
$$
\max\bigg(\sum_{\left|a_{k}\right|>1}
\frac{\left|a_{k}\right|^{2}-1}{\left|a_{k}+1\right|^{2}},\,\,
\sum_{\left|a_{k}\right|<1}\frac{1-\left|a_{k}\right|^{2}}{\left|a_{k}+1\right|^{2}}\bigg)
=n\frac{1+r}{1-r},
\qquad
\mathcal{D}_{2}(a)  = \sum_{\left|a_{k}\right|>1}\frac{\left|a_{k}\right|+1}{\left|a_{k}\right|-1}
=\frac{n}{1-r}.
$$
Thus, by \eqref{Lp_norm_upper_estim_test_func_1},
\[
\max\bigg(\sum_{\left|a_{k}\right|>1}\frac{\left|a_{k}\right|^{2}-1}
{\left|a_{k}+1\right|^{2}},\,\,\sum_{\left|a_{k}\right|<1}
\frac{1-\left|a_{k}\right|^{2}}{\left|a_{k}+1\right|^{2}}\bigg)
\mathcal{D}_{2}^{\frac{1}{p}}(a) = 
\left(n\frac{1+r}{1-r}\right)^{1+\frac{1}{p}} \lesssim
\frac{f'(-1)}{\|f\|_{L^p}}.
\]

In the case $p=1,$ we consider the test function $g$ defined by 
\eqref{testfunc2}. We have
\begin{equation}
\label{deriv_test_func} 
\begin{aligned}
g' & = b_{-r}''\bigg(\sum_{k=0}^{N}b_{-r}^{k}\bigg)^{2}+2\left(b_{-r}'\right)^{2}\bigg(\sum_{k=0}^{N-1}(k+1)b_{-r}^{k}\bigg)\sum_{k=0}^{N}b_{-r}^{k} \\
   & = - \frac{2}{1+rz}b_{-r}'\bigg(\sum_{k=0}^{N}b_{-r}^{k}\bigg)\bigg(r\sum_{k=0}^{N}b_{-r}^{k}+\frac{1-r^{2}}{1+rz}\sum_{k=0}^{N-1}(k+1)b_{-r}^{k}\bigg).
\end{aligned}
\end{equation}
As before, this gives 
$$
g'(-1) =\frac{2r(1+r)}{(1-r)^{2}}(N+1)^{2}+2\left(\frac{1+r}{1-r}\right)^{2}(N+1)\sum_{k=0}^{N-1}(k+1)
\gtrsim \frac{n^3}{(1-r)^2},
$$
which gives the required estimate from below for $g'(-1)/\|g\|_{L^p}$ if we use 
\eqref{Lp_norm_test_func_2}.
\qed


\subsection{\label{subsub_sharp_GenSp_Bern_LpLq} 
Sharpness of (\ref{qgeqp_up})}
Here we show the asymptotic sharpness 
of the constants $\mathcal{C}_{n,\, r}\left(L^{q},\, L^{p}\right)$
as $n\to\infty$ and $r\to 1-$. Clearly, we need to show the sharpness only 
for sufficiently large values of $n$ (for small values of $n$ one may use the test
function $(1-rz)^{-2}$).
\medskip

\textbf{Step 1. The case $1\leq p\le q\leq\infty$.}
We consider the same test function $g$ defined in (\ref{testfunc2}).
Let us estimate from below the norm $\|g'\|_{L^q}$
using the representation \eqref{deriv_test_func} for $g'$.
Taking $b_{-r}(\xi)$ as the new variable (as in the proof of Lemma \ref{bab}), 
we get
$$
\left\Vert g'\right\Vert _{L^{q}}^{q}
=2^q \int_{\mathbb{T}}\left|\frac{1}{1+rb_{-r}(\xi)}\right|^{q}\left|b_{-r}'(b_{-r}(\xi))\right|^{q-1}\bigg\vert D_{N+1}(\xi)\bigg(rD_{N+1}(\xi)+\frac{1-r^{2}}{1+rb_{-r}(\xi)}D_{N+1}'(\xi)\bigg)\bigg\vert^{q}{\rm d}{m}(\xi).
$$
Since $1+rb_{-r}(z)=\frac{1-r^{2}}{1+rz}$ and 
$b_{-r}'\circ b_{-r}(z)=-\frac{(1+rz)^{2}}{1-r^{2}},$
we have 
$$
\begin{aligned}
\left\Vert g'\right\Vert _{L^{q}}^{q} 
& =\frac{2^q}{\left(1-r^{2}\right)^{2q-1}}\int_{\mathbb{T}}\left|1+r\xi\right|^{3q-2}\bigg\vert D_{N+1}(\xi)\bigg(rD_{N+1}(\xi)+(1+r\xi)D_{N+1}'(\xi)\bigg)\bigg\vert^{q}{\rm d}{m}(\xi) \\
& \ge \frac{2^q}{\left(1-r^{2}\right)^{2q-1}}
\int_{I_{2N}}\bigg\vert D_{N+1}(\xi)\bigg(rD_{N+1}(\xi)+
(1+r\xi)D_{N+1}'(\xi)\bigg)\bigg\vert^{q}{\rm d}{m}(\xi),
\end{aligned}
$$
where the arcs $I_{N}$ are defined at the beginning of the section. 
Now, by \eqref{Fosc_result} and \eqref{Lq_norm_D_n}, 
$$
\int_{I_{2N}}\bigg\vert D_{N+1}(\xi)\bigg\vert^{2q}{\rm d}{m}(\xi)
\lesssim N^{2-\frac{1}{q}}
$$
and
\[
\int_{I_{2N}}\bigg\vert(1+r\xi)D_{N+1}(\xi)D_{N+1}'(\xi)\bigg\vert^{q}{\rm d}{m}(\xi)\geq(1+r^{2})^{\frac{q}{2}}\int_{I_{N}}\bigg\vert D_{N+1}(\xi)D_{N+1}'(\xi)\bigg\vert^{q}{\rm d}{m}(\xi)
\gtrsim N^{3-\frac{1}{q}}.
\]
We conclude that 
\begin{equation}
\label{lower_estim_Lq_norm_g_prime}
\left\Vert g'\right\Vert _{L^{q}} 
\gtrsim\left(\frac{1}{1-r^{2}}\right)^{2-\frac{1}{q}}N^{3-\frac{1}{q}}.
\end{equation}
As a result, combining (\ref{lower_estim_Lq_norm_g_prime}) and (\ref{Lp_norm_test_func_2})
we obtain 
$$
\frac{\left\Vert g'\right\Vert _{L^{q}}}{\left\Vert g\right\Vert _{L^{p}}}  
\gtrsim \left(\frac{n}{1-r}\right)^{1+\frac{1}{p}-\frac{1}{q}}
$$
for sufficiently large values of $n.$
\medskip

\textbf{Step 2. The case $1\leq q\leq p\leq\infty$.} We consider now a simpler
test function $h(z)=\frac{1}{1-rz}b_{r}^{n-1}(z)$. We have 
\[
\|h\|_{L^{p}}=\Big\|\frac{1}{1-rz}\Big\|_{L^{p}}\lesssim\frac{1}{(1-r)^{1-\frac{1}{p}}}
\]
with a constant depending on $p$. Furthermore,
$$
h'  =  (n-1)\frac{1}{1-rz}b_{r}'b_{r}^{n-2}+\frac{r}{(1-rz)^{2}}b_{r}^{n-1}
  =  b_{r}'b_{r}^{n-2}\left(\frac{n-1}{1-rz}-\frac{r}{1-r^{2}}b_{r}\right),
$$
and, by the change of variable $\zeta = b_{-r}(\xi)$, 
\begin{eqnarray*}
\left\Vert h'\right\Vert _{L^{q}}^{q} & = & \int_{\mathbb{T}}\left|b_{r}'(\xi)\right|\left|b_{r}'(\xi)\right|^{q-1}\left|\frac{n-1}{1-r\xi}-\frac{r}{1-r^{2}}b_{r}(\xi)\right|{\rm d}{m}(\xi)\\
 & = & \int_{\mathbb{T}}\left|b_{r}'(b_{r}(\zeta))\right|^{q-1}\left|\frac{n-1}{1-rb_{r}(\zeta)}-\frac{r}{1-r^{2}}\zeta\right|^{q}{\rm d}{m}(\zeta)\\
 & = & \frac{1}{(1-r^{2})^{2q-1}}\int_{\mathbb{T}}\left|(1-r\zeta)^{2}\right|^{q-1}\left|(n-1)(1-r\zeta)-r\zeta\right|^{q}{\rm d}{m}(\zeta).
\end{eqnarray*}

Now, integrating over the arc $\frac{\pi}{2} \le \arg \zeta \le \frac{3\pi}{2}$,
it is easily seen that for $n\ge 3$,
$$
\bigg(\int_{\mathbb{T}}\left|(1-r\zeta)^{2}\right|^{q-1}\left|(n-1)(1-r\zeta)-r\zeta
\right|^{q} {\rm d}{m}(\zeta)\bigg)^{1/q} \ge cn,
$$
where $c>0$ is a numerical constant independent of $q$.
Thus, $\left\Vert h'\right\Vert _{L^{q}}\gtrsim n (1-r)^{\frac{1}{q}-2}$, 
$r\in(0,\,1)$, and so 
\[
\frac{\left\Vert h'\right\Vert _{L^{q}}}{\left\Vert h\right\Vert _{L^{p}}}
\gtrsim\frac{n}{(1-r)^{1+\frac{1}{p}-\frac{1}{q}}},
\qquad r\in(0,\,1), \ n\geq 3.
\]
\end{proof}


\subsection{\label{subsub_sharp_Nik}Sharpness of \textup{(\ref{inequ2})}}

Let $g$ be the test function defined in (\ref{testfunc2}).
Recall that, by \eqref{Lp_norm_test_func_2},
$$
\left\Vert g\right\Vert _{L^{p}}\lesssim
\frac{n^{2-\frac{1}{p}}}{\left(1-r\right)^{1-\frac{1}{p}}}, \qquad 1 \le p\leq\infty.
$$
On the other hand, we have (after the change of the variable)
\[
\left\Vert g\right\Vert _{L^{q}}^{q}=\frac{1}{\left(1-r^{2}\right)^{q-1}}\int_{\mathbb{T}}\left|1+r\xi\right|^{2q-2}\left|D_{N+1}(\xi)\right|^{2q}{\rm d}{m}(\xi)
\]
and
\[
\int_{\mathbb{T}}\left|1+r\xi\right|^{2q-2}\left|D_{N+1}(\xi)\right|^{2q}{\rm d}{m}(\xi)\geq 
\int_{I_{N+1}}\left|D_{N+1}(\xi)\right|^{2q}{\rm d}{m}(\xi)
\gtrsim n^{2q-1}.
\]
We conclude that for $1\le p<q\le \infty$,
\[
\frac{\left\Vert g\right\Vert _{L^{q}}}{\left\Vert g\right\Vert _{L^{p}}}
\gtrsim\frac{n^{\frac{1}{p}-\frac{1}{q}}}{\left(1-r\right)
^{\frac{1}{p}-\frac{1}{q}}}.
\]                       
\qed


\end{document}